\documentclass[12pt,reqno]{amsart}
\usepackage{amssymb}
\usepackage{amsthm}
\usepackage{amsmath}
\usepackage[dvips]{epsfig}
\usepackage{graphicx}
\usepackage{hyperref}
\usepackage{color}
\usepackage{url}
\usepackage{tikz}
\usepackage{caption}
\usepackage{subcaption}
\usepackage{epstopdf}
\usepackage[arrow, matrix, curve]{xy}
\usepackage{marginnote}

\usetikzlibrary{decorations.markings}

\makeatletter\@addtoreset {equation}{section}\makeatother

\newtheorem{theorem}{Theorem}
\newtheorem{proposition}{Proposition}[section]
\newtheorem{conjecture}{Conjecture}
\newtheorem{lemma}{Lemma}[section]
\newtheorem{remark}{Remark}[section]
\newtheorem{definition}{Definition}[section]
{\begin{trivlist} \item[]{\em Proof }}%
{\hspace*{\fill}$\rule{.3\baselineskip}{.35\baselineskip}$\end{trivlist}}

\DeclareMathOperator{\csch}{csch}

\DeclareMathOperator{\R}{\mathbb{R}}

\begin{document} 
	
\title[Green's function for the fractional KdV equations]{Green's function for the fractional KdV equation on the periodic domain via Mittag--Leffler's function}

\author{Uyen Le}
\address[U. Le]{Department of Mathematics and Statistics, McMaster University,
	Hamilton, Ontario, Canada, L8S 4K1}
\email{leu@mcmaster.ca}

\author{Dmitry E. Pelinovsky}
\address[D. Pelinovsky]{Department of Mathematics and Statistics, McMaster University,
	Hamilton, Ontario, Canada, L8S 4K1}
\email{dmpeli@math.mcmaster.ca}


\keywords{Fractional Laplacian, Green's function, positivity and monotonicity, periodic domain}


\maketitle

\begin{abstract}
	The linear operator $c + (-\Delta)^{\alpha/2}$, where $c > 0$ and 
	$(-\Delta)^{\alpha/2}$ is the fractional Laplacian on the periodic domain,  arises in the existence of periodic travelling waves in the fractional Korteweg--de Vries equation. We establish a relation of the Green's function of this linear operator with the Mittag--Leffler function, which was previously used in the context of Riemann--Liouville's and Caputo's fractional derivatives. By using this relation, we prove that Green's function is strictly positive and single-lobe (monotonically decreasing away from the maximum point) for every $c > 0$ and every $\alpha \in (0,2]$. On the other hand, we argue from numerical approximations that in the case of $\alpha \in (2,4]$, the Green's function is positive and single-lobe for small $c$ 
	and non-positive and non-single lobe for large $c$. 
\end{abstract}


\section{Introduction}	

This work deals with Green's function for the linear operator 
\begin{equation}
\label{linear-operator}
L_{c,\alpha} := c + (-\Delta)^{\alpha/2}, 
\end{equation}
where $c>0$ is arbitrary parameter and $(-\Delta)^{\alpha/2}$, $\alpha > 0$ is the fractional Laplacian on the normalized periodic domain $\mathbb{T} = [-\pi, \pi]$. The fractional Laplacian is defined via Fourier series by 
\begin{equation}
f(x) = \sum_{n \in \mathbb{Z}} f_n e^{inx}, \quad 
(-\Delta)^{\alpha/2} f(x) = \sum_{n \in \mathbb{Z}} |n|^{\alpha} f_n e^{inx}.
\end{equation}
Properties of the fractional Laplacian on the $d$-dimensional torus $\mathbb{T}^d$ were studied in \cite{RS16}. Recent review of boundary-value problems for the fractional Laplacian and related applications can be found in \cite{JCP}.

Green's function denoted by $G_{\mathbb{T}}$ satisfies the periodic boundary value problem
\begin{equation}
\label{green}
\left[ c + (-\Delta)^{\alpha/2} \right] G_{\mathbb{T}}(x)= \delta(x), \hspace{0.2in}x\in \mathbb{T},
\end{equation}
where $\delta$ is the Dirac delta distribution. The solution is represented via Fourier series by
\begin{equation}\label{green-fourier}
G_{\mathbb{T}}(x) = 
\frac{1}{2\pi}\sum_{n\in \mathbb{Z}}\frac{\cos(nx)}{c+|n|^\alpha}= \frac{1}{2\pi}\left(\frac{1}{c}+ 2\sum_{n=1}^{\infty}\frac{\cos(nx)}{c+n^\alpha}\right).
\end{equation}

Green's function $G_{\mathbb{T}}$ defined by (\ref{green}) and (\ref{green-fourier})  arises in the study of the stationary equation 
\begin{equation}
\label{ode-stat}
\left[ c + (-\Delta)^{\alpha/2} \right] \psi(x) = \psi(x)^{1+p}, \quad x \in \mathbb{T},
\end{equation}
where $p \in \mathbb{N}$. The stationary equation (\ref{ode-stat}) 
defines the travelling periodic waves of the fractional Korteweg--de Vries (fKdV) equation with the speed $c$ \cite{C,BC,hur,J,NPU1,NPU2} 
and the standing periodic waves of the fractional nonlinear Schr\"{o}dinger (fNLS) equation with the frequency $c$ \cite{Johnson,stefanov}. Periodic solutions in other nonlinear elliptic equations associated with the fractional Laplacian were also considered, e.g., in \cite{Ambr,DuGui}.

Green's function $G_{\mathbb{T}}$ defined by (\ref{green}) and (\ref{green-fourier})
was used in the proof of strict positivity of the periodic 
solutions of the stationary equation (\ref{ode-stat}) for $c > 0$, 
$\alpha \in (0,2]$, and $p = 1$ by using Krasnoselskii's fixed point theorem (see Theorem 2.2 in \cite{Peli}). The important ingredient of the proof is the property of strict positivity of Green's function $G_{\mathbb{T}}$ for every $c > 0$.

The property of strict positivity of Green's function was proven for different boundary-value problems associated with the fractional operators in \cite{Nieto} for $\alpha \in (0,1)$ and in \cite{Bai} for $\alpha \in (1,2)$; however, the fractional derivatives were considered in the Riemann--Liouville sense (see \cite{Kilbas,Podlubny} for review of fractional derivatives). 

Here we prove strict positivity 
of Green's function $G_{\mathbb{T}}$ satisfying the boundary-value problem 
(\ref{green}) on $\mathbb{T}$ for every $c > 0$ and every $\alpha \in (0,2]$.
Moreover, we show that $G_{\mathbb{T}}$ 
has the single-lobe profile in the sense that 
$G_{\mathbb{T}}$ is monotonically decreasing on $\mathbb{T}$ away from its maximum point located at $x = 0$. The following theorem presents this result.

\begin{theorem}
	\label{main}
For every $c > 0$ and every $\alpha \in (0,2]$, Green's function $G_{\mathbb{T}}$ defined by (\ref{green}) and (\ref{green-fourier}) is 
even, strictly positive on $\mathbb{T}$, and monotonically decreasing on $(0,\pi)$. 
\end{theorem}

The result of Theorem \ref{main} is known in the context of 
Green's function $G_{\mathbb{R}}$ for the linear operator $\mathcal{L}_{c,\alpha}$ in  (\ref{linear-operator}) considered on the real line $\mathbb{R}$ (see Lemma A.4 in \cite{F-L}). This was shown 
from similar properties of the heat kernel related to 
the fractional Laplacian $(-\Delta)^{\alpha/2}$ (see Lemma A.1 in \cite{F-L}). The constant $c > 0$ in $\mathcal{L}_{c,\alpha}$ can be normalized to unity 
when $\mathcal{L}_{c,\alpha}$ is considered on the real line $\mathbb{R}$. 

The same properties hold for Green's function $G_{\mathbb{T}}$ 
on the periodic domain $\mathbb{T}$ because it can be written 
as the following periodic superposition of Green's function $G_{\mathbb{R}}$ 
on the real line $\mathbb{R}$:
\begin{equation}
G_{\mathbb{T}}(x) = \sum_{n \in \mathbb{Z}} G_{\mathbb{R}}(x-2\pi n), \quad x \in \mathbb{T}.
\end{equation}
Hence, if $G_{\mathbb{R}}(x) > 0$ for $x \in \mathbb{R}$, then 
$G_{\mathbb{T}}(x) > 0$ for $x \in \mathbb{T}$ and if 
$G_{\mathbb{R}}'(x) \leq 0$ for $x \geq 0$, then $G_{\mathbb{T}}'(x) \leq 0$ 
for $x \in [0,\pi]$. Here the parameter $c$ in $G_{\mathbb{T}}$ cannot 
be normalized to unity. 

The main novelty of our work is the relation between Green's function $G_{\mathbb{T}}$ and the Mittag--Leffler function \cite{M-L}. 
The Mittag--Leffler function naturally arises in the other (Riemann--Liouville and Caputo) formulations 
of fractional derivatives \cite{Kilbas,Podlubny} 
but it has not been used in the context of the fractional 
Laplacian to the best of our knowledge. In particular, 
we prove Theorem \ref{main} for $\alpha \in (0,2)$ by using the integral representations and properties of the Mittag--Leffler function 
and some trigonometric series from \cite{Prudnikov}. For $\alpha = 2$, 
the result of Theorem \ref{main} can be readily shown by writing $G_{\mathbb{T}}$ in the exact analytical form (see Appendix \ref{appendix-a}).

Figure \ref{profile} illustrates the statement of Theorem \ref{main}. 
It shows the single-lobe positive profile of $G_{\mathbb{T}}$ for two values of $c$ in the case $\alpha = 0.5$ (left) and $\alpha = 1.5$ (right). The only  difference between these two cases is that $G_{\mathbb{T}}(0)$ is bounded for $\alpha > 1$ and is unbounded for $\alpha \leq 1$.

\begin{figure}[h]
	\centering
	\includegraphics[width=0.49\linewidth]{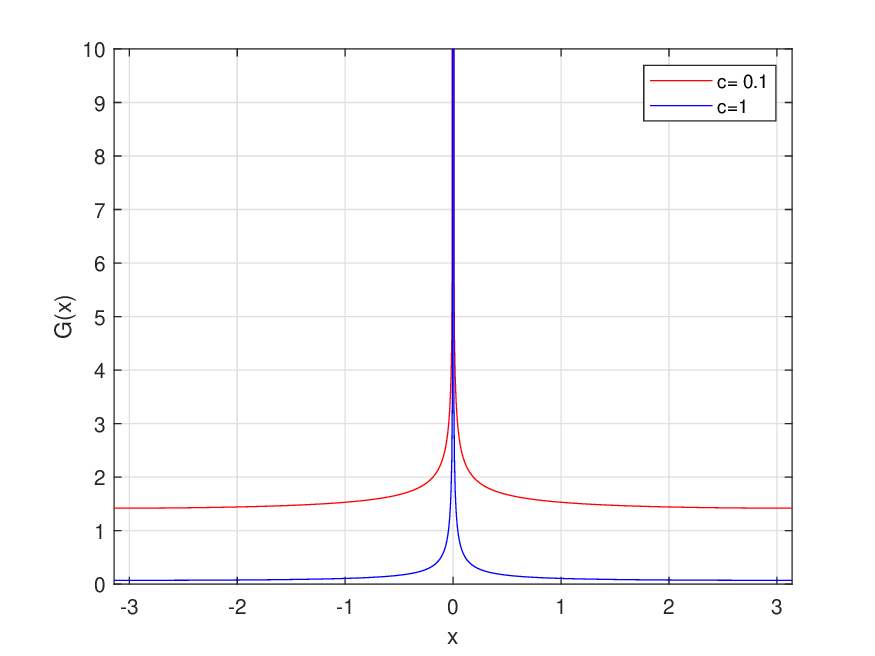}
	\includegraphics[width=0.49\linewidth]{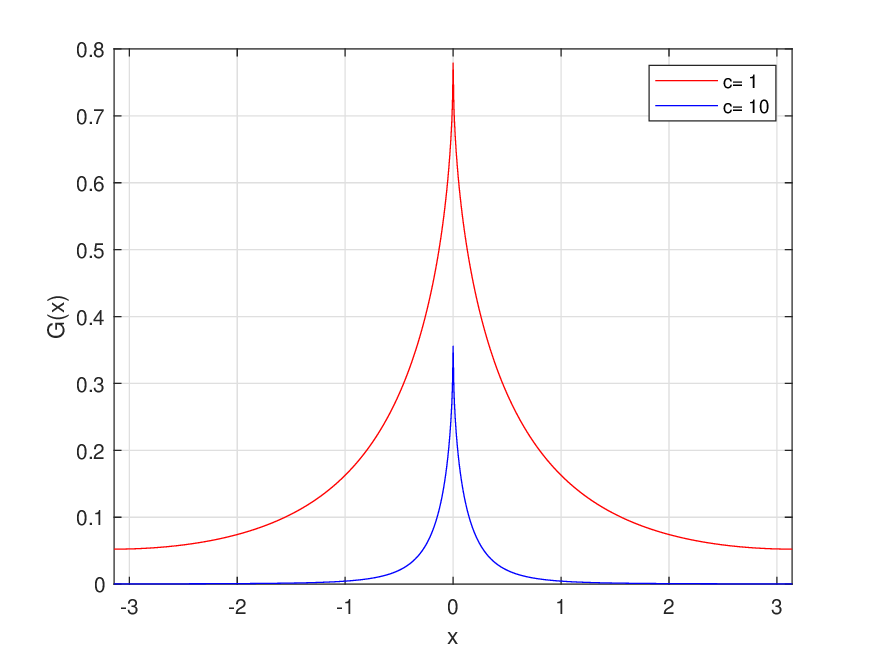}
	\caption{Profiles of $G_{\mathbb{T}}$ for $\alpha =0.5$ (left) and $\alpha = 1.5$ (right) for specific values of $c$.}
	\label{profile}
\end{figure}

Green's function $G_{\mathbb{R}}$ on the real line $\mathbb{R}$ is also used to understand interactions of strongly localized waves, e.g. attractive versus repelling interactions \cite{GO,Manton} (see also \cite{CP,Parker}). These interactions were recently studied in \cite{Dem1,Dem2} in the context of the beam equation, which corresponds to the case $\alpha = 4$. The fractional cases for $\alpha \in (0,2)$ and $\alpha \in (2,4)$ are also important from applications in quantum computing, fluid dynamics, 
and elasticity theory.

Theorem \ref{main} shows that the properties of $G_{\mathbb{T}}$ for $\alpha \in (0,2)$ are similar to those for $\alpha = 2$ (the same holds for $G_{\mathbb{R}}$). However, it is an open question if the properties of 
$G_{\mathbb{T}}$ for $\alpha \in (2,4)$ are similar to those for $\alpha = 4$, 
for which $G_{\mathbb{R}}$ has infinitely many oscillations, whereas the number of oscillations of $G_{\mathbb{T}}$ depends on $c > 0$ and becomes infinite in the limit of $c \to \infty$ (see Appendix B). In the second part of this paper, 
we present numerical results which support the following conjecture.

\begin{conjecture}\label{conj-main}
For each $\alpha \in (2,4]$, there exists $c_0 > 0$ such that 
for $c \in (0,c_0)$, Green's function $G_{\mathbb{T}}$ defined by (\ref{green}) and (\ref{green-fourier}) is even, strictly positive on $\mathbb{T}$, and monotonically decreasing on $(0,\pi)$. For $c \in [c_0,\infty)$, $G_{\mathbb{T}}$ has a finite number of zeros on $\mathbb{T}$. The number of zeros is bounded in the limit of $c \to \infty$ if $\alpha \in (2,4)$ and unbounded as $c \to \infty$ if $\alpha = 4$.
\end{conjecture}

Since the limit $c \to \infty$ for Green's function $G_{\mathbb{T}}$ can be rescaled as Green's funciton $G_{\mathbb{R}}$ with $c$ normalized to unity, 
Conjecture \ref{conj-main} implies the following conjecture (relevant for  interactions of strongly localized waves in \cite{Dem1,Dem2}).

\begin{conjecture}\label{conj-line}
	For every $c > 0$ and every $\alpha \in (2,4]$, Green's function $G_{\mathbb{R}}$ is not strictly positive on $\mathbb{R}$ and is not monotonically decreasing on $(0,\infty)$. It has a finite number of zeros on $\mathbb{R}$ if $\alpha \in (2,4)$ and an infinite number of zeros if $\alpha = 4$.
\end{conjecture}

Figure \ref{profile_super} illustrates the statement of Conjecture \ref{conj-main}. For $\alpha = 2.5$ (top), Green's function $G$ has the single-lobe positive profile for $c = 2$ (red curve) but it is not positive 
for $c = 10$ (blue curve). For $\alpha = 3.5$ (bottom), it is positive for $c = 1$ (red curve), has one pair of zeros for $c = 10$ (blue curve), and has two pairs of zeros for $c = 60$ (black curve).

\begin{figure}[h]
	\centering
	\includegraphics[width=0.7\linewidth]{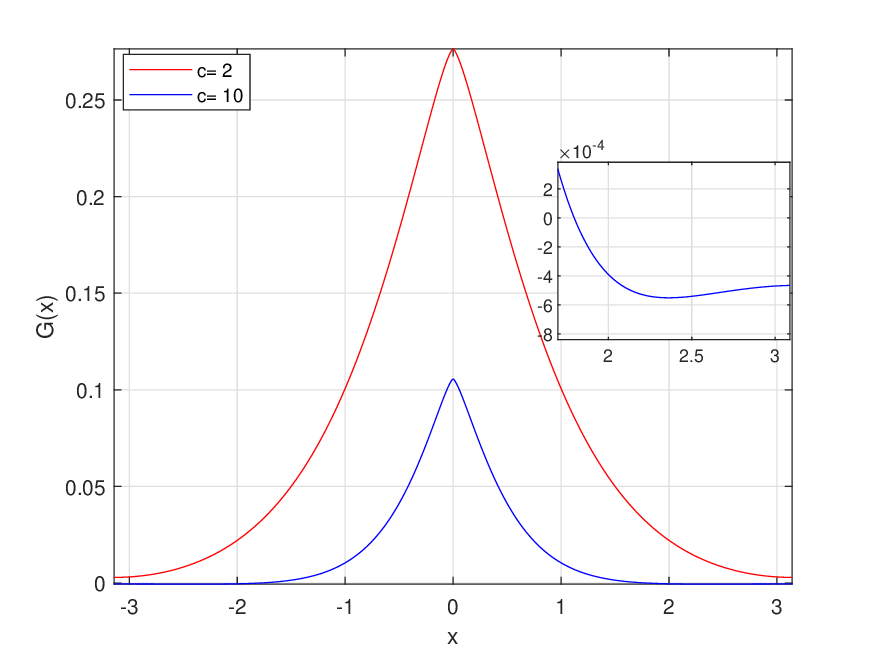}\\
	\includegraphics[width=0.49\linewidth]{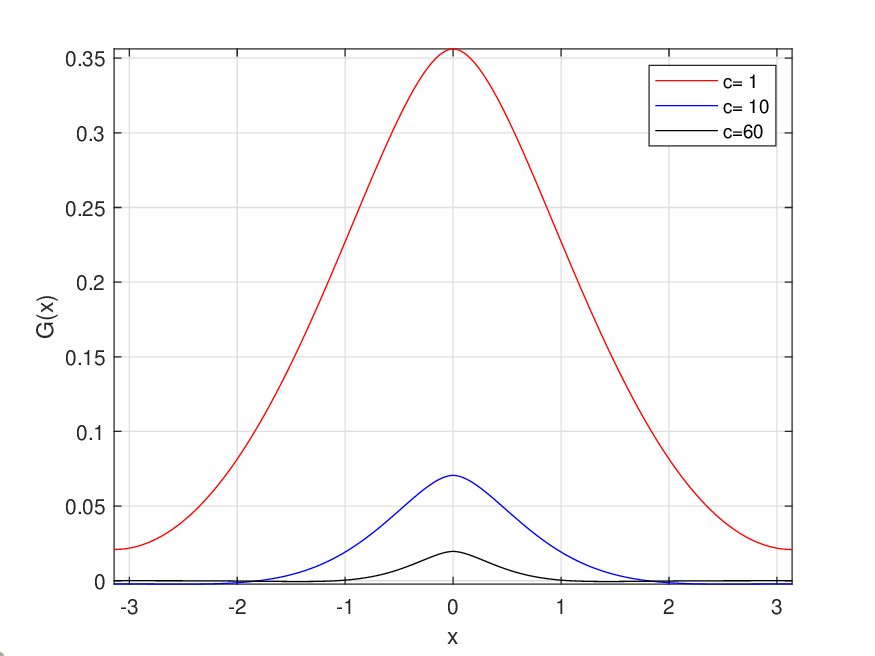}
	\includegraphics[width=0.49\linewidth]{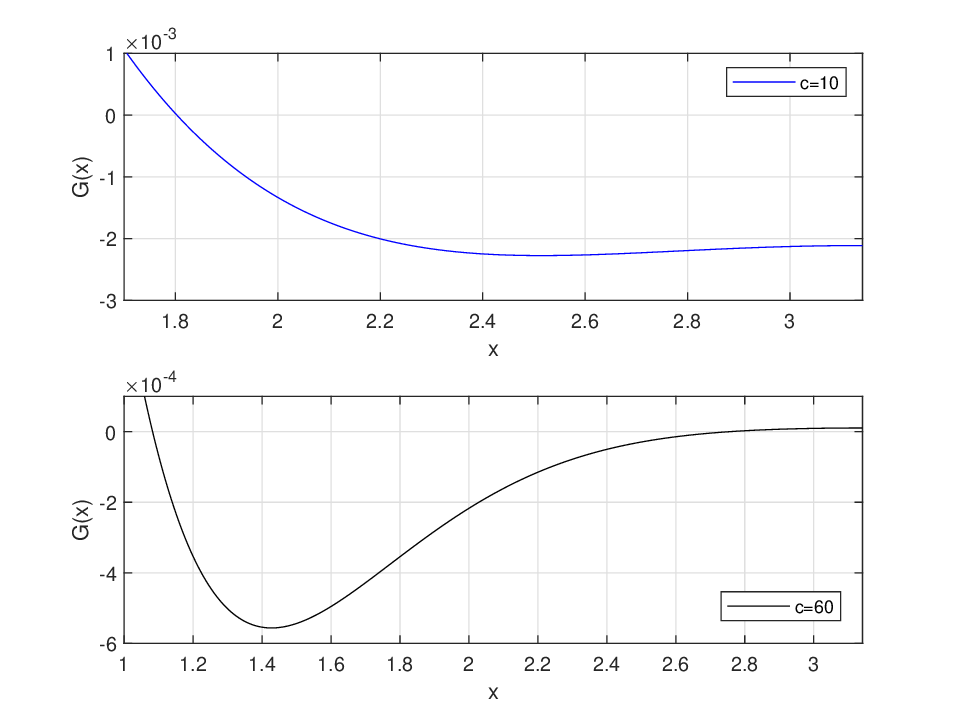}
	\caption{ Profiles of $G$ on $\mathbb{T}$ for $\alpha=2.5$ (top) and  $\alpha =3.5$ (bottom) at specific values of $c$.}
	\label{profile_super}
\end{figure}

The remainder of the article is structured as follow. Section \ref{ML} presents an overview of the Mittag-Leffler function and its properties. The integral representation of Green's function $G_{\mathbb{T}}$ is derived in Section \ref{int-rep}. The proof of Theorem \ref{main} is presented in Section \ref{mainproof}. The validity of Conjecture \ref{conj-main} is discussed 
in Section \ref{conclusion}. Conclusion is given in Section \ref{sec-conclusion}.  
Appendices \ref{appendix-a} and \ref{appendix-b} give explicit formulas for Green's function $G_{\mathbb{T}}$
for local cases of $\alpha = 2$ and $\alpha = 4$, respectively. 
Appendix \ref{appendix-c} contains formal asymptotic results 
to support Conjecture \ref{conj-main} for $\alpha > 2$ with small $|\alpha -2|$.

\section{Properties of Mittag--Leffler function}
\label{ML}

Here we discuss properties of the Mittag--Leffler function defined by 
	\begin{equation}
	\label{mlf}
		E_{\alpha}(x)= \sum_{k=0}^{\infty}\frac{x^k}{\Gamma(k\alpha+1)}, \qquad \alpha > 0,
	\end{equation}
and its two-parametric generalization defined by 
\begin{equation}
\label{mlf-2}
E_{\alpha, \beta}(x)= \sum_{k=0}^{\infty}\frac{x^k}{\Gamma(k\alpha+\beta)}, \qquad \alpha, \beta > 0.
\end{equation}

Mittag--Leffler functions were introduced in the theory of analytic functions \cite{M-L}. In recent years, they became popular due to their applications in fractional differential equations \cite{Kilbas}. Indepth studies of the Mittag--Leffler functions can be found in \cite{Mittag3} and \cite{Mittag2}. 

Mittag--Leffler functions are typically used to represent solutions of 
initial-value problems for the fractional differential equations defined by the  Riemann-Liouville or Caputo fractional derivatives \cite{Kilbas}. As our study involves the boundary-value problem for the fractional Laplacian $(-\Delta)^{\alpha/2}$, Mittag--Leffler function $E_{\alpha}(-x^\alpha)$ 
is used in the integral representation of Green's function $G_{\mathbb{T}}$. This integral representation is derived in Section \ref{int-rep}.

Here we review some important properties of the Mittag--Leffler functions.

	\begin{lemma}
		\label{deriv E}
	For every $\alpha >0$ and every $x\in \R$, it is true that 
\begin{equation}
\label{E-der}
E_{\alpha, \alpha}(x)= \alpha \frac{d}{dx}E_{\alpha}(x).
\end{equation}
\end{lemma}

\begin{proof}
	The result is obtained by differentiating (\ref{mlf}) and using 
	(\ref{mlf-2}):
	\begin{equation*}
	\frac{d}{dx}E_{\alpha}(x)= \sum_{k=1}^{\infty}\frac{x^{k-1}}{\alpha\Gamma(\alpha k)} = \frac{1}{\alpha}\sum_{k=0}^{\infty}\frac{x^k}{\Gamma(\alpha k+\alpha)} = \frac{1}{\alpha} E_{\alpha,\alpha}(x).
	\end{equation*}
The series converges absolutely for every $x \in \mathbb{R}$ since $E_{\alpha}$ and $E_{\alpha,\alpha}$ are entire functions. 
\end{proof}

\begin{lemma}\cite{Pollard}
\label{complete monotonic}
For every $\alpha\in (0,1]$, 
	the function $x \mapsto E_{\alpha}(-x)$ 
	is positive and completely monotonic for $x \geq 0$, that is
	\begin{equation}
	\label{E-positive}
	(-1)^m\frac{d^m}{dx^m}E_{\alpha}(-x)\ge 0, \qquad m\in \mathbb{N}, \quad 
	x \geq 0.
	\end{equation}
Consequently, $E_{\alpha,\alpha}(-x) \geq 0$ for every $x \geq 0$.
\end{lemma}

\begin{remark} 
	A necessary and sufficent condition for the function $x \mapsto E_{\alpha}(-x)$  to be completely monotonic for $x \geq 0$ is that $E_{\alpha}(-x)$ can be expressed in the form 
 	\begin{equation*}
	E_{\alpha}(-x)= \int_{0}^{\infty}e^{-xt}dF_{\alpha}(t), \quad x \geq 0,
	\end{equation*}
	where $F_{\alpha}$ is a nondecreasing and bounded on $(0,\infty)$.
The proof of \cite{Pollard} is based on the representation of $E_{\alpha}(-x)$  given by 
	\begin{equation*}
	E_{\alpha}(-x)= \frac{1}{2i\pi \alpha}\int_{C} \frac{e^{t^{1/\alpha}}}{t+x}dt,
	\end{equation*}	
with a specially selected the contour $C$ in $\mathbb{C}$.
\end{remark}

\begin{lemma}\cite{Mittag2}
\label{asymptotic-mlf}	
For every $\alpha \in (0,2)$, $E_{\alpha}(-x^{\alpha})$ admits the asymptotic 
expansion
\begin{equation}
\label{mfl-asymp-1}
E_{\alpha}(-x^{\alpha}) = -\sum_{k=1}^N \frac{(-1)^k}{\Gamma(1-\alpha k) x^{\alpha k}} + \mathcal{O}\left(\frac{1}{|x|^{\alpha N + \alpha}}\right) \quad 
\mbox{\rm as} \quad x \to \infty,
\end{equation}
where $N \in \mathbb{N}$ is arbitrarily fixed. For every $\alpha \geq 2$, $E_{\alpha}(-x^{\alpha})$ admits the asymptotic 
expansion
\begin{equation}
\label{mfl-asymp-2}
E_{\alpha}(-x^{\alpha}) = \frac{1}{\alpha} 
\sum_{n = -N+1}^N e^{a_n x} + \mathcal{O}\left(\frac{1}{|x|^{\alpha}}\right) \quad 
\mbox{\rm as} \quad x \to \infty,
\end{equation}
where $a_n = e^{\frac{i\pi(2n-1)}{\alpha}}$ and $N$ is the largest integer satisfying the bound $2N-1 \leq \frac{\alpha}{2}$. 
\end{lemma}

\begin{remark}
Asymptotic expansions (\ref{mfl-asymp-1}) and (\ref{mfl-asymp-2}) can be differentiated term by term.
\end{remark}

\begin{remark}
	\label{rem-explicit}
	We list the explicit cases of the Mittag--Leffler function $E_{\alpha}(-x^{\alpha})$ for the first integers:
	\begin{eqnarray*}
		\alpha = 1, &&\quad E_{1}(-x)= e^{-x}, \\
		\alpha = 2, &&\quad E_{2}(-x^2)= \cos(x), \\
		\alpha = 3, &&\quad  E_{3}(-x^3)= \frac{1}{3} e^{-x} + \frac{2}{3} \displaystyle{e^{\frac{x}{2}}} \cos\left(\frac{\sqrt{3}x}{2}\right), \\	
		\alpha = 4, &&\quad E_{4}(-x^4)= \cos\left(\frac{x}{\sqrt{2}}\right) 
		\cosh\left(\frac{x}{\sqrt{2}}\right).
	\end{eqnarray*}
For $\alpha = 1$, the asymptotic representation (\ref{mfl-asymp-1}) admits zero leading-order terms for every $N \in \mathbb{N}$. The asymptotic representation (\ref{mfl-asymp-2}) is also obvious from the exact expressions for $\alpha = 2,3,4$, moreover, the remainder term 
	is zero for $\alpha = 2$ and can be included to the summation by increasing $N$ by one for $\alpha = 3$ and $\alpha = 4$.
\end{remark}

\begin{lemma}\cite{Mittag2}
	\label{integralE}
	For every $\alpha \in (0, 2)$ and every $x\in \R$, $E_{\alpha}(-x)$ satisfies the following integral representation,
	\begin{equation}
	\label{integral-repr}
	E_{\alpha}(-x^\alpha)= \frac{2}{\pi}\sin\left(\frac{\pi\alpha}{2}\right)\int_{0}^{\infty}\frac{t^{\alpha-1}\cos(xt)}{1+2t^{\alpha}\cos\left(\frac{\pi\alpha}{2}\right)+t^{2\alpha}}dt.
	\end{equation}
\end{lemma}
\begin{remark}
	It is claimed in \cite{Mittag2} that the integral representation (\ref{integral-repr}) is true for all $\alpha >0$, however, 
	the integral is singular for $\alpha = 2$ and a discrepancy exists 
	at $x = 0$ for $\alpha > 2$. For example, when $\alpha =3$, it follows 
	from (\ref{mlf}) that $E_{3}(0)=1$ whereas computing the integral given in \eqref{integral-repr} via the change of variable $u = t^3$ gives
	\begin{align*}
	E_{3}(0)&=-\frac{2}{3\pi}\int_{0}^{\infty}\frac{du}{1+u^2}= -\frac{1}{3}\ne 1.
	\end{align*} 
	Hence, the integral representation (\ref{integral-repr}) can only be used 
	for $\alpha \in (0,2)$, for which $E_{\alpha}(-x^{\alpha})$ is bounded and 
	decaying as $x \to +\infty$. 
\end{remark}

\section{Integral representation of Green's function $G_{\mathbb{T}}$}
\label{int-rep}

Here, we take Green's function $G_{\mathbb{T}}$ defined by the Fourier series in (\ref{green-fourier}) and rewrite it in the integral form 
involving the Mitag--Leffler function $E_{\alpha, \alpha}$. 
The following proposition gives the result for $\alpha \in (0,2]$.

\begin{proposition}
	\label{prop-1}
For every $c > 0$ and every $\alpha \in (0,2]$, it is true that 
\begin{equation}
\label{G-int}
G_{\mathbb{T}}(x) = \frac{1}{2\pi c} + \frac{1}{\pi c} \int_{0}^{\infty}\left(\frac{e^t\cos(x)-1}{1-2e^t\cos(x)+e^{2t}}\right)t^{\alpha-1}E_{\alpha, \alpha}(-ct^\alpha)dt, \quad x \in \mathbb{T}.
\end{equation}
\end{proposition}

\begin{proof}
	Assume first that $x \neq 0$ and $c \in (0,1)$. Expanding each term of the trinometric sum in (\ref{green-fourier}) into absolutely convergent geometric series and interchanging the two series, we obtain  
\begin{equation}
\label{sum-integral-0}
\sum_{n=1}^{\infty}\frac{\cos(nx)}{c+n^\alpha}
=\sum_{n=1}^{\infty}\frac{\cos(nx)}{n^\alpha}\sum_{k=0}^{\infty}\left(\frac{-c}{n^\alpha}\right)^k= \sum_{k=0}^{\infty}(-c)^k\sum_{n=1}^{\infty}\frac{\cos(nx)}{n^{\alpha(k+1)}}.
\end{equation}
It is known from the integral representation (1) in \cite[Section 5.4.2]{Prudnikov} that for every $x \ne 0$ and $\alpha > 0$ that 
\begin{equation}
\label{sum-integral}
\sum_{n=1}^{\infty}\frac{\cos(nx)}{n^{\alpha(k+1)}}= \frac{1}{\Gamma(\alpha k+ \alpha)}\int_{0}^{\infty}\frac{t^{\alpha(k+1)-1}\left(e^t\cos(x)-1\right)}{1-2e^t\cos(x)+e^{2t}}dt,
\end{equation}
where $k \geq 0$. Substituting \eqref{sum-integral} into (\ref{sum-integral-0}) 
and interchanging formally the summation and the integration yields the 
following representation:
\begin{align}
\sum_{n=1}^{\infty}\frac{\cos(nx)}{c+n^\alpha} &= \sum_{k=0}^{\infty}\frac{(-c)^k}{\Gamma(\alpha k+\alpha)}\int_{0}^{\infty}\frac{t^{\alpha(k+1)-1}\left(e^t\cos(x)-1\right)}{1-2e^t\cos(x)+e^{2t}}dt,
\label{technical-expansions} \\
&= \int_{0}^{\infty}\left(\frac{e^t\cos(x)-1}{1-2e^t\cos(x)+e^{2t}}\right)t^{\alpha-1}\sum_{k=0}^{\infty}\frac{\left(-ct^\alpha\right)^k}{\Gamma(\alpha k+ \alpha)} dt,\\
&= \int_{0}^{\infty}\left(\frac{e^t\cos(x)-1}{1-2e^t\cos(x)+e^{2t}}\right)t^{\alpha-1}E_{\alpha, \alpha}(-ct^\alpha)dt.
\end{align} 
This yields formally the integral formula (\ref{G-int}). Let us now justify 
the interchange of summation and integration in (\ref{technical-expansions}).
Using the chain rule and Lemma \ref{deriv E}, we get
\begin{equation}
\label{chain-rule}
t^{\alpha-1}E_{\alpha, \alpha}(-ct^\alpha)= -\frac{1}{c}\frac{d}{dt}E_{\alpha}(-ct^\alpha).
\end{equation}
It follows from (\ref{chain-rule}) that for every 
$\alpha \in (0,2]$, the asymptotic expansion 
(\ref{mfl-asymp-1}) in Lemma \ref{asymptotic-mlf} for $\alpha \in (0,2)$ 
and Remark \ref{rem-explicit} for $\alpha = 2$ imply that 
\begin{equation}
\sup_{t \in [0,\infty)} t^{\alpha - 1} |E_{\alpha,\alpha}(-t^{\alpha})| < \infty.
\label{mtf-value-1}
\end{equation}
Hence, the integral in (\ref{G-int}) converges absolutely for every $x \neq 0$ and $\alpha \in (0,2]$. Similarly, the integral in (\ref{technical-expansions}) 
converges absolutely for every $x \neq 0$ and $\alpha \in (0,2]$, whereas 
the numerical series converges absolutely for every $c \in (0,1)$. Thus, 
the interchange of summation and integration in 
(\ref{technical-expansions}) is justified by Fubini's theorem. 

For $x = 0$, we note that $G_{\mathbb{T}}(0) < \infty$ if $\alpha > 1$ and 
$G_{\mathbb{T}}(0) = \infty$ if $\alpha \in (0,1]$. Since $E_{\alpha,\alpha}(-x^{\alpha}) = 1 + \mathcal{O}(x^{\alpha})$ as $x \to 0$, the integral in (\ref{G-int}) 
converges absolutely for $x = 0$ and $\alpha \in (1,2]$ and diverges 
for $x = 0$ and $\alpha \in (0,1]$. Hence, the integral representation 
(\ref{G-int}) holds again for $x = 0$, $c \in (0,1)$, and $\alpha \in (0,2]$.

In order to extend the integral representation (\ref{G-int}) from $c \in (0,1)$ to every $c > 0$, we use real analyticity of Green's function $G_{\mathbb{T}}$ and the integral in (\ref{G-int}) in $c$ for $c > 0$. Due to uniqueness of the analytical continuation of both $G_{\mathbb{T}}$ and the integral in (\ref{G-int}) in $c$, the equality in (\ref{G-int}) is uniquely continued from $c \in (0,1)$ to $c > 0$.
\end{proof}

The integral representation (\ref{G-int}) 
of Green's function $G_{\mathbb{T}}$ can be justified 
for $\alpha > 2$ provided that $c$ is sufficiently small. This result is 
described by the following proposition.

\begin{proposition}
	\label{prop-2}
For every $\alpha > 2$, there exists $c_{\alpha} > 0$ given by 
\begin{equation}
\label{c-alpha}
c_{\alpha} := \left[ \cos\left(\frac{\pi}{\alpha}\right) \right]^{-\alpha},
\end{equation}
such that for every $c \in (0,c_{\alpha})$, the integral 
representation (\ref{G-int}) is true for every $x \in \mathbb{T}$.
\end{proposition}

\begin{proof}
The asymptotic expansion (\ref{mfl-asymp-2}) in Lemma \ref{asymptotic-mlf} 
implies for every $c > 0$ and $\alpha > 2$ that 
	\begin{equation}
	\sup_{t \in [0,\infty)} e^{-t \cos\left(\frac{\pi}{\alpha}\right)} t^{\alpha - 1}  |E_{\alpha,\alpha}(-t^{\alpha})| < \infty,
	\label{mtf-value-2}
	\end{equation}
where we have used again the connection formula (\ref{chain-rule}).
In addition, $E_{\alpha,\alpha}(-x^{\alpha}) = 1 + \mathcal{O}(x^{\alpha})$ 
as $x \to 0$. Due to the above properties, the integral in (\ref{G-int}) converges absolutely for every $x \in \mathbb{T}$ if $c \in (0,c_{\alpha})$, where $c_{\alpha}$ is given by (\ref{c-alpha}). This justifies the formal computations in the proof of Proposition \ref{prop-1}.
\end{proof}

\begin{remark}
	For $c  \geq c_{\alpha}$ and $\alpha > 2$, the Fourier series representation (\ref{green-fourier}) suggests that $|G_{\mathbb{T}}(x)| < \infty$ for every $x \in \mathbb{T}$. However, the integral in (\ref{G-int}) does not converge absolutely, hence it is not clear if the integral representation (\ref{G-int}) can be used in this case. Our numerical results in Section \ref{conclusion} show that the integral representation (\ref{G-int}) 
	cannot be used for $c > c_{\alpha}$.
\end{remark}

\section{Green's function $G_{\mathbb{T}}$ for $\alpha \in (0,2)$}
\label{mainproof}

Here, we prove Theorem \ref{main} by using the integral representation 
(\ref{G-int}) in terms of the Mittag--Leffler function $E_{\alpha,\alpha}$.
It follows from \eqref{green-fourier} that $G_{\mathbb{T}}$ is even for every $c > 0$ and $\alpha >0$. Furthermore, if $\alpha \in (0,1]$, then $\lim\limits_{x\to 0} G_{\mathbb{T}}(x)= +\infty$, and if $\alpha >1$, then
\begin{equation*}
G_{\mathbb{T}}(0)= \frac{1}{2\pi}\left(\frac{1}{c}+ 2\sum_{n=1}^{\infty}\frac{1}{c+n^\alpha}\right)>0.
\end{equation*}
We shall prove that $G_{\mathbb{T}}'(x) \leq 0$ for $x \in (0,\pi)$ and 
$G_{\mathbb{T}}(\pi) > 0$ 
for every $c > 0$ and $\alpha \in (0,2]$. For $\alpha = 2$, this result follows from 
the exact analytical representation of $G_{\mathbb{T}}$ in Appendix \ref{appendix-a}. 
Therefore, we focus on the case $\alpha \in (0,2)$ here.
The following proposition gives an integral representation for $G_{\mathbb{T}}(\pi)$ 
which implies its strict positivity for every $c > 0$ and $\alpha \in (0,2)$

\begin{proposition}
	\label{prop-3}
	For every $c > 0$ and every $\alpha \in (0,2)$, it is true that 
	\begin{equation}
	\label{G-pi}
G_{\mathbb{T}}(\pi) =\frac{\sin(\frac{\alpha\pi}{2})}{\pi c^{1-\frac{1}{\alpha}}} \int_{0}^{\infty} \frac{s^\alpha \csch(\pi c^{\frac{1}{\alpha}}s)}{1+ 2s^\alpha \cos(\frac{\alpha\pi}{2}) + s^{2\alpha}}   ds,
	\end{equation}
which implies $G_{\mathbb{T}}(\pi) > 0$. 
\end{proposition}

\begin{proof}
Evaluating the integral representation \eqref{G-int} at $x= \pi$, we obtain
\begin{equation}\label{G-pi1}
G_{\mathbb{T}}(\pi)= \frac{1}{2\pi c} - \frac{1}{\pi c} \int_{0}^{\infty}\frac{1}{1+e^t}t^{\alpha-1}E_{\alpha, \alpha}(-ct^\alpha)dt.
\end{equation}
Substituting (\ref{chain-rule}) into \eqref{G-pi1}, integrating by parts, 
and using the asymptotic representation (\ref{mfl-asymp-1}) to get zero contribution in the limit of $t \to \infty$, we obtain 
\begin{equation}\label{G-pi2}
G_{\mathbb{T}}(\pi)= \frac{1}{\pi c} \int_{0}^{\infty}\frac{e^t}{(1+e^t)^2}E_{\alpha}(-ct^\alpha)dt,
\end{equation} 	
where the integral converges absolutely for every $c > 0$ and $\alpha \in (0,2)$. Substituting the integral representation (\ref{integral-repr}) for  $E_{\alpha}(-c t^{\alpha})$ from Lemma \ref{integralE} into \eqref{G-pi2}, we obtain 
\begin{equation}
\label{interchange}
G_{\mathbb{T}}(\pi)=  \frac{2}{\pi^2 c} \sin\left(\frac{\alpha\pi}{2}\right)
\int_{0}^{\infty} \frac{e^t}{(1+e^t)^2}\int_{0}^{\infty}\frac{s^{\alpha-1}\cos(c^{\frac{1}{\alpha}}ts)}{1+2s^\alpha\cos(\frac{\pi \alpha}{2})+ s^{2\alpha}}  dsdt.
\end{equation}
Since both integrands belong to $L^1(0, \infty)$, the order of integration in (\ref{interchange}) can be interchanged to get
\begin{equation}
\label{interchange-2}
G_{\mathbb{T}}(\pi) = \frac{2}{\pi^2 c} \sin\left(\frac{\alpha\pi}{2}\right) 
\int_{0}^{\infty}\frac{s^{\alpha-1}}{1+2s^\alpha\cos(\frac{\pi \alpha}{2})+ s^{2\alpha}}\int_{0}^{\infty}\frac{e^t\cos(c^{\frac{1}{\alpha}}st)}{(1+e^t)^2}dtds.
\end{equation}
The inner integral is evaluated exactly with the help of 
integral (7) in \cite[Section 2.5.46]{Prudnikov}:
\begin{align*}
\int_{0}^{\infty}\frac{e^t\cos(c^{\frac{1}{\alpha}}st)}{(1+e^t)^2}dt =   \frac{\pi}{2} c^{\frac{1}{\alpha}} s \csch(\pi c^{\frac{1}{\alpha}}s),
\end{align*}
When it is substituted into (\ref{interchange-2}), it yields the 
integral representation (\ref{G-pi}). The integrand is positive 
and absolutely integrable for every $c > 0$ and $\alpha \in (0,2)$, 
which implies that $G_{\mathbb{T}}(\pi) > 0$.
\end{proof}

\begin{remark}
	Positivity of $G_{\mathbb{T}}(\pi)$ for $c > 0$ and $\alpha \in (0,1]$ also follows from 
	the representation (\ref{G-pi2}) due to positivity of $E_{\alpha}(-ct^{\alpha})$ for every $t > 0$ 
	in Lemma \ref{complete monotonic}. However, 
	$E_{\alpha}(-ct^{\alpha})$ is not positive for all $t>0$ when $\alpha>1$, 
	hence, the representation (\ref{G-pi2}) is not sufficient for the proof 
	of positivity of $G_{\mathbb{T}}(\pi)$ if $\alpha \in (1,2)$.
\end{remark}

It remains to prove that $G_{\mathbb{T}}'(x) \leq 0$ for every $x \in (0,\pi)$. The proof is carried differently for $\alpha \in (0,1]$ and for $\alpha\in (1,2)$.  
In the former case, we obtain the integral representation for $G_{\mathbb{T}}'(x)$, 
which is strictly negative for $x \in (0,\pi)$. In the latter case, we employ the variational method to verify that the unique solution $G_{\mathbb{T}}$ of 
the boundary-value problem \eqref{green} admits the single lobe profile, with the only maximum located at the point of symmetry at $x=0$. The following two propositions give these two results.

\begin{proposition}
	\label{prop-4}
	For every $c > 0$ and every $\alpha \in (0,1]$, 
$G_{\mathbb{T}}'(x) < 0$ for every $x \in (0,\pi)$. 
\end{proposition}

\begin{proof}
Differentiating the integral representation \eqref{G-int} in $x$ yields
\begin{align}
\nonumber
G_{\mathbb{T}}'(x)&=\frac{1}{\pi c}\int_{0}^{\infty}t^{\alpha-1}E_{\alpha, \alpha}(-ct^{\alpha})\frac{d}{dx}\left(\frac{e^t\cos(x)-1}{1-2e^t\cos(x)+e^{2t}}\right)dt,\\
\label{G-der-pos}
&= -\frac{\sin(x)}{\pi c}\int_{0}^{\infty}t^{\alpha-1}E_{\alpha, \alpha}(-ct^{\alpha})\frac{e^t(e^{2t}-1)}{\left(1-2e^t\cos(x)+e^{2t}\right)^2}dt,
\end{align} 
where the integrand is absolutely integrable. 
It follows by Lemma \ref{complete monotonic} that $E_{\alpha, \alpha}(-ct^\alpha)\ge 0$ for $t > 0$. Since $\sin(x)> 0$ for $x\in(0,\pi)$, and the integrand is positive, it follows from the integral representation (\ref{G-der-pos}) that  $G_{\mathbb{T}}'(x)<0$ for $x\in (0,\pi)$. 
\end{proof} 

\begin{proposition}
	\label{prop-5}
	For every $c > 0$ and every $\alpha \in (1,2)$, 
	$G_{\mathbb{T}}'(x) \leq 0$ for every $x \in (0,\pi)$. 
\end{proposition}

\begin{proof}
The proof consists of the following two steps. First, we obtain a variational solution to the boundary-value problem \eqref{green}. Second, we use the fractional Polya--Szeg\"o inequality to show that the solution $G_{\mathbb{T}}$ has a single-lobe profile on $\mathbb{T}$ with the only maximum located at the point of symmetry at $x=0$.

\textbf{Step 1}: Let us consider the following minimization problem,
\begin{equation}\label{minimizer}
\mathcal{B}_{c}:=\min_{u\in H^{\frac{\alpha}{2}}_{\rm per}(\mathbb{T})}\lbrace B_{c}(u)- u(0)\rbrace,
\end{equation}
where the quadratic functional $B_{c}(u)$ is given by
\begin{equation}
\label{bilinear}
B_{c}(u)= \frac{1}{2} 
\int_{\mathbb{T}} \left[ \left(D^{\frac{\alpha}{2}}u\right)^2 +cu^2 \right] dx.
\end{equation}
Since $c>0$, we have 
\begin{align*}
\frac{1}{2}\min(1, c)\|u\|_{ H^{\frac{\alpha}{2}}_{\rm per}(\mathbb{T})}&\le B_{c}(u) \le \frac{1}{2}\max(1, c)\|u \|_{H^{\frac{\alpha}{2}}_{\rm per}(\mathbb{T})},
\end{align*}
hence, $B_{c}(u)$ is equivalent to the squared $H^{\frac{\alpha}{2}}_{\rm per}(\mathbb{T})$ norm. Moreover, for $\alpha\in (1, 2)$, $\delta \in H^{-\frac{\alpha}{2}}_{\rm per}(\mathbb{T})$, the dual of $H^{\frac{\alpha}{2}}_{\rm per}(\mathbb{T})$ since 
	\begin{equation*}
	\|\delta\|_{H^{-\frac{\alpha}{2}}_{\rm per}(\mathbb{T})}=\sum_{\xi\in \mathbb{Z}}\frac{1}{\left(1+ |\xi|^2\right)^{\frac{\alpha}{2}}} < \infty.
	\end{equation*}
Thus, by Lax--Milgram theorem (see Corollary 5.8 in \cite{Brezis}), there exists a unique $G_{\mathbb{T}} \in H^{\frac{\alpha}{2}}_{\rm per}(\mathbb{T})$ such that $G_{\mathbb{T}}$ is the global minimizer of the variational problem \eqref{minimizer}, 
for which the Euler--Lagrange equation is equivalent to the boundary-value problem (\ref{green}). By uniqueness of solutions of the two problems, $G_{\mathbb{T}}$ is equivalently written as the Fourier series (\ref{green-fourier}), from which it follows that $G_{\mathbb{T}}(\pi) < G_{\mathbb{T}}(0)$. Hence, $G_{\mathbb{T}}$ is different from a constant function on $\mathbb{T}$.

\begin{remark}
	The variational method and in particular the Lax--Milgram theorem cannot be applied to the case $\alpha \in (0,1]$ since the Dirac delta distibution $\delta$ does not belong to the dual space of $H^{\frac{\alpha}{2}}_{\rm per}(\mathbb{T})$ when $\alpha \in (0,1]$.
\end{remark}

\textbf{Step 2}: We utilize the fractional Polya--Szeg\"o inequality, proved in the appendix of \cite{Johnson}, to show that a symmetric decreasing rearrangement of the minimizer $G_{\mathbb{T}}$ on $\mathbb{T}$ does not increase $B_c(u)$. For completeness, we state the following definition and lemma.

\begin{definition}
	Let $m$ be the Lebesgue measure on $\mathbb{T}$ and $f(x): \R\to \R$ be a $2\pi$ periodic function. The symmetric and decreasing rearrangement $\tilde {f}$ of $f$ on $\mathbb{T}$ is given by
	\begin{equation}
	\tilde{f}(x)= \inf\lbrace t : \;\; m(\lbrace z\in \mathbb{T} : \;\; 
	f(z)>t \rbrace) \le 2|x|\rbrace,  \hspace{0.2in} x\in \mathbb{T}. 
	\end{equation}
	\label{def-sym}
\end{definition}
The rearrangement $\tilde{f}$ satisfies the following properties:
\begin{itemize}
	\item [i)] $\tilde{f}(-x) = \tilde{f}(x)$ and $f'(x) \leq 0$ for $x \in (0, \pi)$.
	\item [ii)]$\tilde{f}(0)= \max_{x\in \mathbb{T}}f(x)$.
	\item [iii)] $\|\tilde{f}\|_{L^2(\mathbb{T})}= \|f\|_{L^2(\mathbb{T})}$.
\end{itemize}

\begin{lemma} \cite{Johnson} \label{Polya}
	For every $\alpha > 1$ and every $f\in H^{\frac{\alpha}{2}}_{per}(\mathbb{T})$, it is true that 
	\begin{equation}
	\int_{-\pi}^{\pi}|D^{\frac{\alpha}{2}}\tilde{f}|^2dx \le \int_{-\pi}^{\pi}|D^{\frac{\alpha}{2}}f|^2dx. 
	\end{equation}
	\label{lem-PS}
\end{lemma}

The argument of the proof in the second step goes as follows. Suppose $\widetilde{G}_{\mathbb{T}}$ is the symmetric and  decreasing rearrangement of $G_{\mathbb{T}}$, then by Lemma \ref{Polya} and by property (iii) of Definition \ref{def-sym} we have $B_{c}(\widetilde{G}_{\mathbb{T}}) \le B_{c}(G_{\mathbb{T}})$.
Since the global minimizer of the variational problem \eqref{minimizer} is uniquely given by $G_{\mathbb{T}}$, $\widetilde{G}_{\mathbb{T}}$ coincides with $G_{\mathbb{T}}$ up to a translation on $\mathbb{T}$. However, it follows from (\ref{green-fourier}) that $G_{\mathbb{T}}(-x) = G_{\mathbb{T}}(x)$ and $G_{\mathbb{T}}(\pi) < G_{\mathbb{T}}(0)$, hence an internal maximum at $x_0 \in (0,\pi)$ would contradicts to the single-lobe profile of $G_{\mathbb{T}}$ and the only maximum of $G_{\mathbb{T}}$ 
is located at $0$, so that $G_{\mathbb{T}}(x) = \widetilde{G}_{\mathbb{T}}(x)$ for every $x \in \mathbb{T}$. It follows from property (i) of Definition \ref{def-sym} that $G_{\mathbb{T}}'(x) \leq 0$ for $x \in (0,\pi)$. 
\end{proof}

\section{Green's function $G_{\mathbb{T}}$ for $\alpha > 2$}
\label{conclusion}

Here we provide numerical approximations of the Green's function $G_{\mathbb{T}}$ for $\alpha > 2$, which support Conjecture \ref{conj-main}.
The profiles of $G_{\mathbb{T}}$ are depicted on Figure \ref{profile_super}. 
We only give details on how zeros of $G_{\mathbb{T}}(\pi)$ depend on parameters $(c,\alpha)$.

It follows from the Fourier series \eqref{green-fourier} that $G_{\mathbb{T}}(\pi)$ can be computed by the numerical series
\begin{equation}
\label{G-1}
G_{\mathbb{T}}(\pi) = \frac{1}{2\pi} \left(\frac{1}{c} + 2 \sum_{n = 1}^{\infty} \frac{(-1)^n}{c + n^{\alpha}} \right),
\end{equation}
where the series converges absolutely if $\alpha > 1$. On the other hand, 
$G_{\mathbb{T}}(\pi)$ can also be computed from the integral representation 
(\ref{G-pi2}), that is, 
\begin{equation}\label{G-2}
G_{\mathbb{T}}(\pi)= \frac{1}{\pi c} \int_{0}^{\infty}\frac{e^t}{(1+e^t)^2}E_{\alpha}(-ct^\alpha)dt,
\end{equation} 	
which converges absolutely for $c \in (0,c_{\alpha})$, see Proposition \ref{prop-2}, where $c_{\alpha}$ is given by (\ref{c-alpha}). 

\begin{figure}[htpb]
	\centering
	\includegraphics[width=0.49\linewidth]{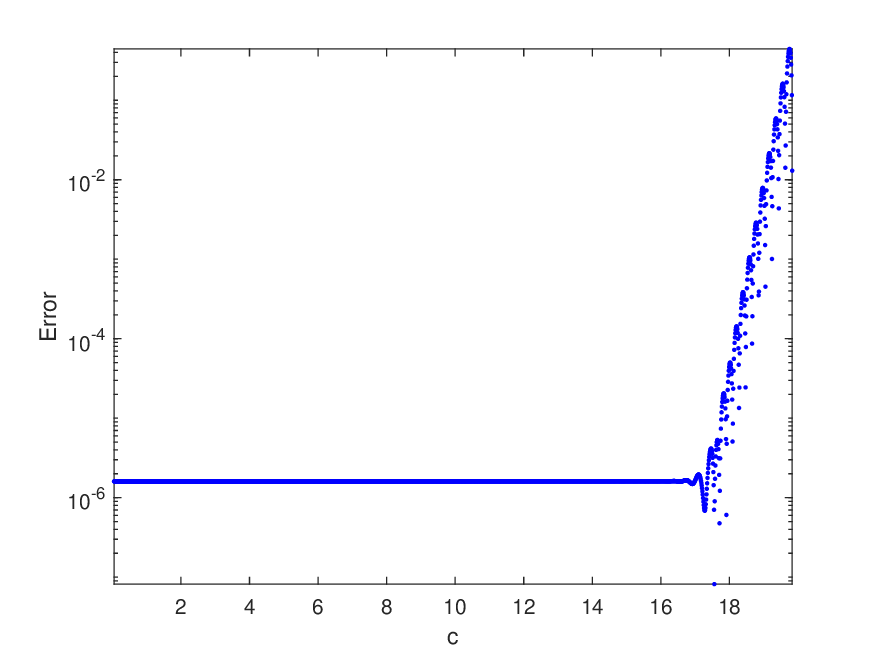}
	\includegraphics[width=0.49\linewidth]{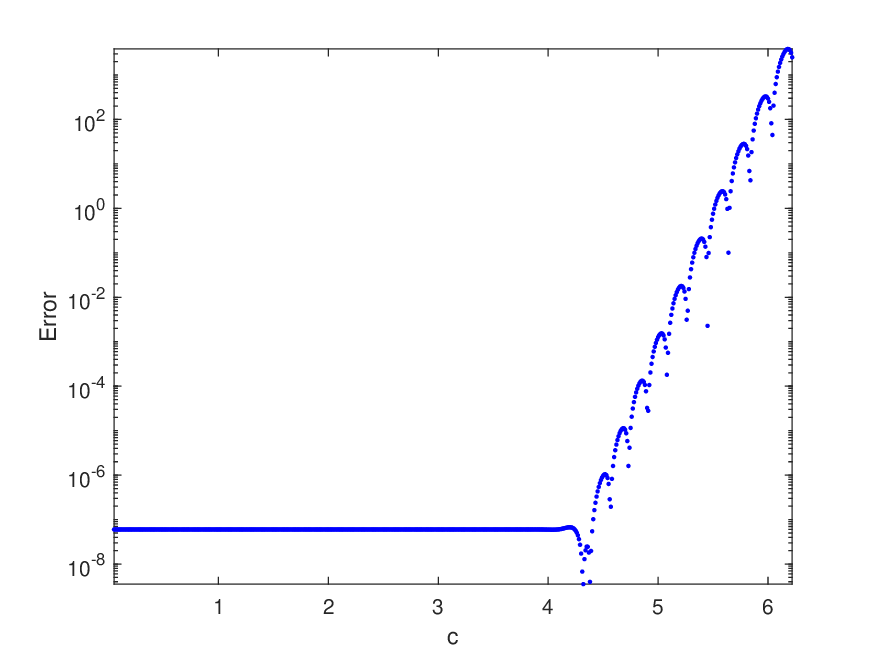}
	\caption{Difference between computations of $G_{\mathbb{T}}(\pi)$ in (\ref{G-1}) and (\ref{G-2}) for $\alpha = 2.5$ (left) and $\alpha =3.5$ (right) versus parameter $c$.}
	\label{pointwise error}
\end{figure}

Figure \ref{pointwise error} shows the difference of $G_{\mathbb{T}}(\pi)$ computed from (\ref{G-1}) and (\ref{G-2}) for $\alpha = 2.5$ (left) and $\alpha =3.5$ (right) in logarithmic scale versus parameter $c$. The Fourier series (\ref{G-1}) is truncated such that the remainder is 
of the size $\mathcal{O}(10^{-10})$. For the integral representation of $G_{\mathbb{T}}(\pi)$ in (\ref{G-2}), we numerically compute the Mittag--Leffler function $E_{\alpha}(-ct^{\alpha})$ on the half line; this task is accomplished by using the Matlab code provided in \cite{Matlab}, where the Mittag-Leffler functions are approximated with relative errors of 
the size $\mathcal{O}(10^{-15})$. As follows from Fig. \ref{pointwise error}, the difference between the two computations is constantly small if $c < c_{\alpha}$, when the integral representation (\ref{G-2}) converges absolutely, where $c_{\alpha = 2.5} \approx 18.8$
and $c_{\alpha = 3.5} \approx 5.2$. However, the accuracy of numerical computations based on the integral representation (\ref{G-2}) deteriorates 
for $c$ approaching $c_{\alpha}$ and as a result, the difference between two computations quickly grows for $c > c_{\alpha}$. 

\begin{figure}[htpb]
	\centering
	\includegraphics[width=0.6\linewidth]{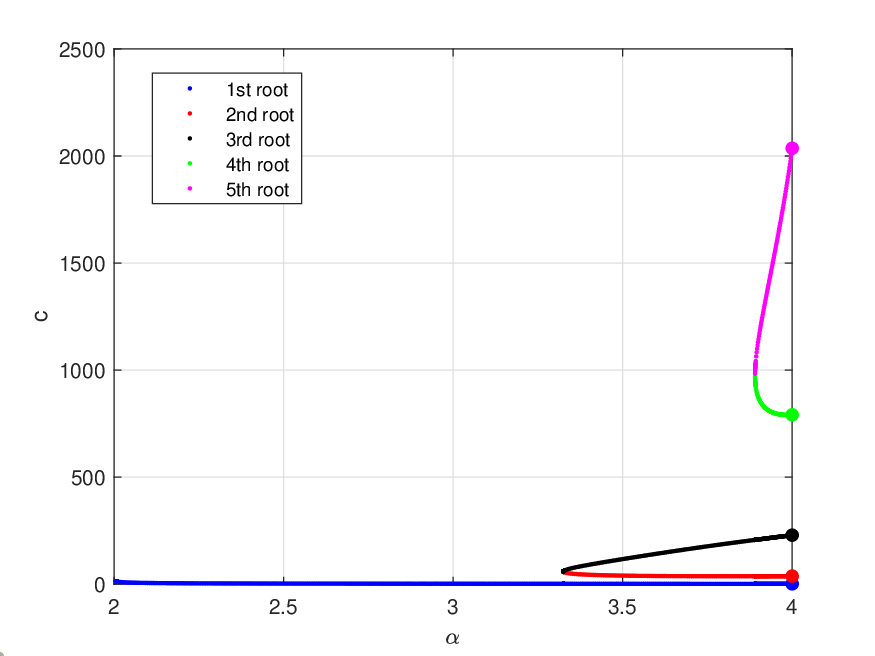} \\
	\includegraphics[width=0.7\linewidth]{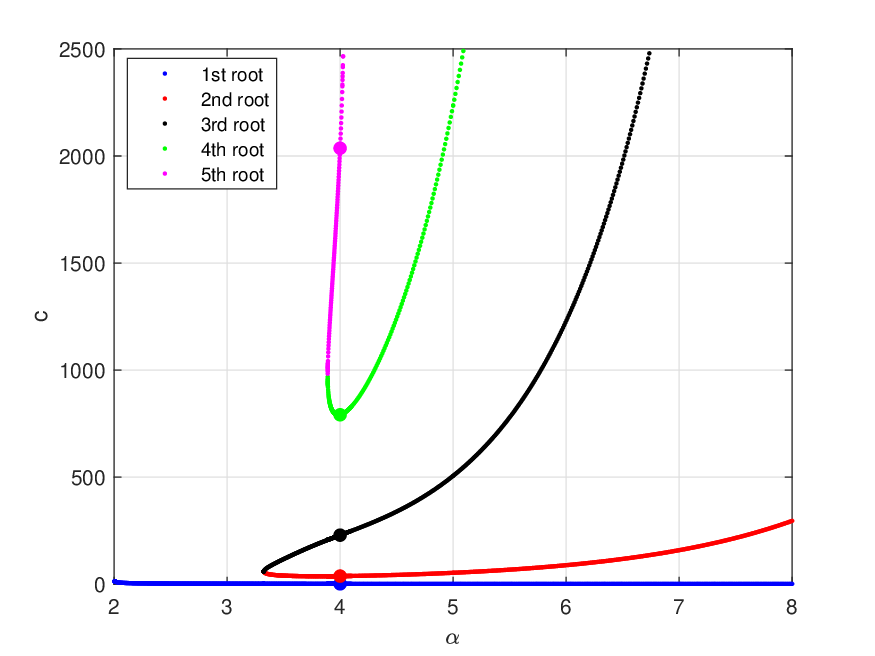}\\
	\includegraphics[width=0.49\linewidth]{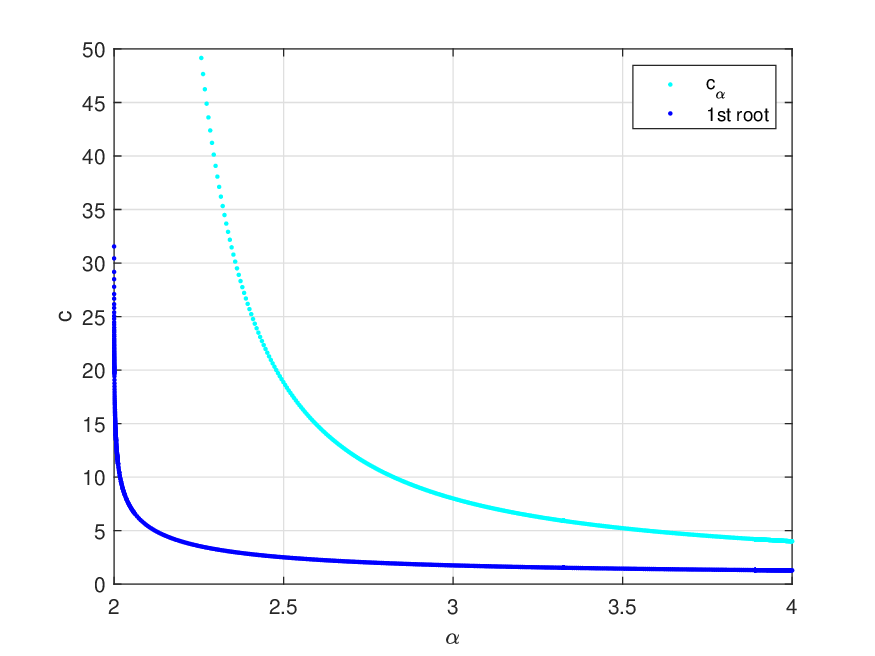}
	\includegraphics[width=0.49\linewidth]{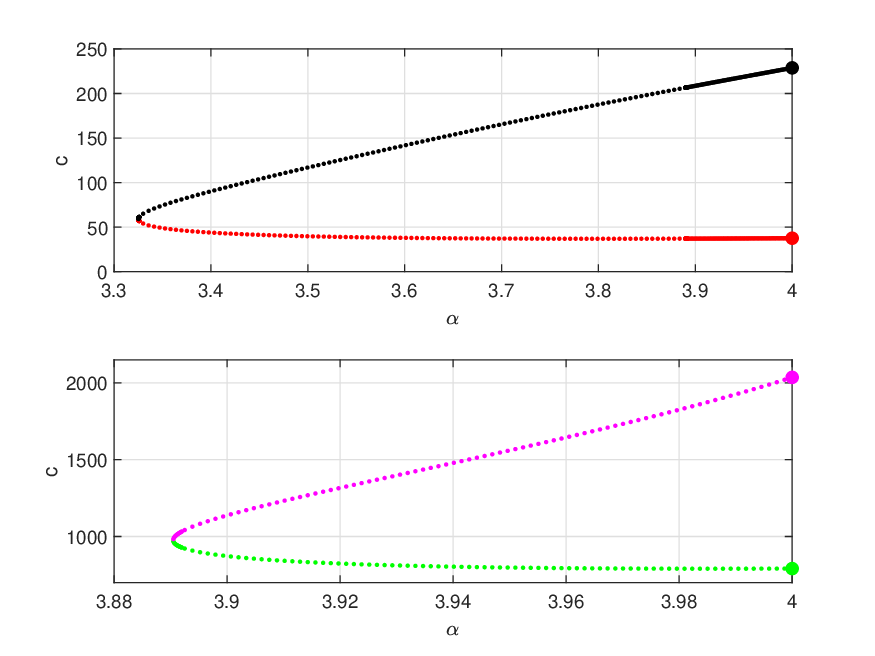}
	\caption{Top: Location of the first five roots of $G_{\mathbb{T}}(\pi)$ on the  $(c,\alpha)$ plane. Bottom: The first root of $G_{\mathbb{T}}(\pi)$ relative to the boundary $c_\alpha$ (left). Coalescence of the 2nd and 3rd roots (upper right) and the 4th and 5th roots (lower right).}
	\label{roots}
\end{figure}

Roots of $G_{\mathbb{T}}(\pi)$ in $c$ for each fixed $\alpha > 2$ are computed from the Fourier series representation (\ref{G-1}) using the bisection method. 
Figure \ref{roots} (top) shows the first five zeros of $G_{\mathbb{T}}(\pi)$ on the $(c,\alpha)$ plane, where the dots show the roots of $G_{\mathbb{T}}(\pi)$ computed from the exact solutions in Appendix \ref{appendix-b} for $\alpha = 4$. 
The first root exists for every $\alpha > 2$ and is located inside $(0,c_{\alpha})$, see the bottom left panel. The other roots are located outside $(0,c_\alpha)$ and disappear via pairwise coalescence as $\alpha$ is reduced towards $\alpha = 2$, see the bottom right panels. The 2nd and 3rd roots coalesce at $\alpha \approx 3.325$ and the 4th and 5th roots coalesce at $\alpha \approx 3.89$. The number of terms in the Fourier series of $G_{\mathbb{T}}(\pi)$ is increased to compute the 4th and 5th roots such that the remainder is of the size of $\mathcal{O}(10^{-14})$ because $G_{\mathbb{T}}(\pi)$ becomes very small near the location of these roots. 

Table \ref{roots-error} compares the error between the numerically detected roots at $\alpha = 4$ and the roots of $G_{\mathbb{T}}(\pi)$ obtained from solving the transcendental equation \eqref{transcendental-eq} in Appendix \ref{appendix-b}.

\begin{table}[ht]
	\centering
	\begin{tabular}{|c|c|}
		Root & Error \\
		\hline
		1st & 1.9915 e-11 \\
		2nd & 7.1495 e-08 \\
		3rd & 3.3182 e-06 \\
		4th & 0.0031\\
		5th & 0.0156 \\
	\end{tabular}
	\caption{Difference between locations of the first five roots of $G_{\mathbb{T}}(\pi)$
		for $\alpha =4$ computed from (\ref{G-1}) and (\ref{transcendental-eq}).}
	\label{roots-error}
\end{table}

Green's function $G_{\mathbb{T}}$ was computed versus $x$ using the Fourier series representation (\ref{green-fourier}) for fixed values of $(c,\alpha)$. 
The plots of $G_{\mathbb{T}}$ are shown in Figures \ref{profile} and \ref{profile_super}. 
Since the first root of $G_{\mathbb{T}}(\pi)$ occurs at $c\approx 2.507$ for $\alpha =2.5$ and at $c \approx 1.446$ for $\alpha =3.5$, see Fig. \ref{roots} (bottom left panel), the threshold $c_0$ in Conjecture \ref{conj-main} reduces with the larger value of $\alpha$. Appendix \ref{appendix-c} 
gives a formal asymptotic approximation of the threshold $c_0$ as $\alpha \to 2$.

\section{Conclusion}
\label{sec-conclusion}

The main contribution of this work is the novel relation between Green's function for the linear operator $c + (-\Delta)^{\alpha/2}$ on the periodic domain $\mathbb{T}$ and the Mittag--Leffler function. With the help of this relation, we have proved that Green's funciton is strictly positive on $\mathbb{T}$ and single-lobe (monotonically decreasing away from the maximum point) for every $c > 0$ and $\alpha \in (0,2]$. 
The same property is also true for sufficiently small $c$ and $\alpha \in (2,4]$ 
but we give numerical and asymptotic results that Green's function 
has a finite number of zeros on $\mathbb{T}$ for sufficiently large $c$, 
the number of zeros is bounded in the limit $c \to \infty$ for $\alpha \in (2,4)$ but is unbounded for $\alpha = 4$. Rigorous proof of properties 
of Green's function for $\alpha \in (2,4)$ is an open problem left for 
further studies.

\appendix
\section{Green's function $G_{\mathbb{T}}$ for $\alpha = 2$}
\label{appendix-a}

Here we derive the exact analytic form of Green's function $G_{\mathbb{T}}$ 
for $\alpha = 2$. The following proposition reproduces Theorem \ref{main} for $\alpha = 2$.

\begin{proposition}
	\label{prop-appendix-a}
	For every $c > 0$, Green's function $G_{\mathbb{T}}$ at $\alpha = 2$ is 
	even, strictly positive on $\mathbb{T}$, and strictly 
	monotonically decreasing on $(0,\pi)$. 	
\end{proposition}

\begin{proof}
For $\alpha=2$, Green's function $G_{\mathbb{T}}$ satisfies the second-order differential equation
\begin{equation}\label{alpha2-ode}
-G_{\mathbb{T}}''(x) + cG_{\mathbb{T}}(x) = \delta(x), \quad x \in \mathbb{T},
\end{equation}
where $c > 0$. It follows from the theory of Dirac delta distributions that $G_{\mathbb{T}}$ is continuous, even,  periodic on $\mathbb{T}$, and have a jump discontinuity of the first derivative at $x = 0$. 

To see the jump condition of $G_{\mathbb{T}}'(x)$ across $x=0$, we integrate \eqref{alpha2-ode} on $(-\varepsilon, \varepsilon)$ and then take the limit as $\varepsilon \to 0$.
\begin{equation}
\label{lim-der}
\lim\limits_{\varepsilon\to 0} \int_{-\varepsilon}^{\varepsilon}\left(-G_{\mathbb{T}}''(x) + cG_{\mathbb{T}}(x)\right)dx = 
\lim\limits_{\varepsilon\to 0} \int_{-\varepsilon}^{\varepsilon} \delta(x) dx = 1,
\end{equation}
where the last equality follows from properties of $\delta$.
Since $G_{\mathbb{T}} \in C^0(\mathbb{R})$, the second term on the left hand side vanishes as $\varepsilon\to 0$, which yields
$-G_{\mathbb{T}}'(0^+)+ G_{\mathbb{T}}'(0^-)= 1$. Since $G_{\mathbb{T}}$ is even on $\mathbb{R}$, we obtain \begin{equation}
\label{G-0-prime} 
G_{\mathbb{T}}'(0^+) = -\frac{1}{2}.
\end{equation} 
Additionally, it follows from the Fourier series representation (\ref{green-fourier}) with $\alpha = 2$ that 
\begin{equation}
\label{G-0}
G_{\mathbb{T}}(0)= \frac{1}{2\pi} \sum_{n\in \mathbb{Z}}\frac{1}{c+ n^2}= \frac{\coth(\sqrt{c}\pi)}{2\sqrt{c}},
\end{equation}
where we have used numerical series (4) in \cite[Section 5.1.25]{Prudnikov}. 

The differential equation 
(\ref{alpha2-ode}) is solved for even $G_{\mathbb{T}}$ as follows:
\begin{eqnarray*}
G_{\mathbb{T}}(x) = G_{\mathbb{T}}(0) \cosh(\sqrt{c} x) + G_{\mathbb{T}}'(0^+) \frac{\sinh(\sqrt{c}|x|)}{\sqrt{c}}, \quad x \in \mathbb{T}.
\end{eqnarray*}
Due to (\ref{G-0-prime}) and (\ref{G-0}), this can be rewritten in the closed form as 
\begin{eqnarray}
	G_{\mathbb{T}}(x) = \frac{\cosh(\sqrt{c}(\pi-|x|))}{2 \sqrt{c} \sinh(\sqrt{c} \pi)}, \quad x \in \mathbb{T}.
\label{G2exact}
\end{eqnarray}
It follows from (\ref{G2exact}) that 
\begin{eqnarray}
G_{\mathbb{T}}'(x) = -\frac{\sinh(\sqrt{c}(\pi-x))}{2\sinh(\sqrt{c} \pi)}<0, \quad x \in (0,\pi),
\label{G2der}
\end{eqnarray}
and hence $G_{\mathbb{T}}$ is strictly monotonically decreasing on $(0,\pi)$. 
On the other hand, 
\begin{eqnarray}
G_{\mathbb{T}}(\pi) = \frac{1}{2\sinh(\sqrt{c} \pi)} >0, \quad c > 0,
\label{G2value}
\end{eqnarray}
and hence $G_{\mathbb{T}}$ is strictly positive on $\mathbb{T}$. Note that the exact expression for $G_{\mathbb{T}}(\pi)$ in (\ref{G2value}) also follows from numerical series (6) in \cite[Section 5.1.25]{Prudnikov}.
\end{proof}

\begin{remark}
It follows from (\ref{G2exact}) that $G_{\mathbb{T}}'(\pi) = 0$, due to smoothness and periodicity of even $G_{\mathbb{T}}(x)$ across $x = \pm \pi$. Therefore, the exact expression in (\ref{G2exact}) and the relation for $G_{\mathbb{T}}(0)$ in (\ref{G-0}) can be alternatively found by solving the differential equation 
	(\ref{alpha2-ode}) for even $G_{\mathbb{T}}$ subject to the boundary conditions 
	$G_{\mathbb{T}}'(0^{\pm}) = \mp \frac{1}{2}$ and $G_{\mathbb{T}}'(\pm \pi) = 0$. 
\end{remark}

\section{Green's function $G_{\mathbb{T}}$ for $\alpha = 4$}
\label{appendix-b}

Here we derive the exact analytic form of Green's function $G_{\mathbb{T}}$ for $\alpha = 4$. The following proposition proves Conjecture \ref{conj-main} for $\alpha = 4$.

\begin{proposition}
	\label{prop-appendix-b}
	There exists $c_0 > 0$ such that for $c \in (0,c_0)$, Green's function $G_{\mathbb{T}}$ at $\alpha = 4$ is 
	even, strictly positive on $\mathbb{T}$, and strictly monotonically decreasing on $(0,\pi)$. For $c \in [c_0,\infty)$, $G_{\mathbb{T}}$ has a finite number of zeros on $\mathbb{T}$, which becomes unbounded as $c \to \infty$.
\end{proposition}

\begin{proof}
For $\alpha=4$, Green's function $G_{\mathbb{T}}$ satisfies the fourth-order differential equation
\begin{equation}\label{alpha4-ode}
G_{\mathbb{T}}''''(x) + cG_{\mathbb{T}}(x) = \delta(x), \quad x \in \mathbb{T},
\end{equation}
where $c > 0$. It follows from the theory of Dirac delta distributions that $G_{\mathbb{T}}$ is continuous, even,  periodic on $\mathbb{T}$, and have a jump discontinuity of the third derivative at $x = 0$. Similarly to the computation in (\ref{lim-der}), it follows that Green's function 
solves the boundary-value problem with the boundary conditions 
\begin{equation}
\label{bc-Green}
G_{\mathbb{T}}'(0) = G_{\mathbb{T}}'(\pm \pi) = G_{\mathbb{T}}'''(\pm \pi) = 0, \quad G_{\mathbb{T}}'''(0^{\pm}) = \pm \frac{1}{2}.
\end{equation}
Due to the boundary conditions (\ref{bc-Green}), it is easier to solve 
the differential equation (\ref{alpha4-ode}) for $G_{\mathbb{T}}'$ on $[0,\pi]$. By using the parametrization $c = 4 a^4$, we obtain
\begin{eqnarray*}
G_{\mathbb{T}}'(x) &=& c_1 \cosh(ax) \cos(ax) + c_2 \cosh(ax) \sin(ax) \\
&& 
+ c_3 \sinh(ax) \cos(ax) + c_4 \sinh(ax) \sin(ax), \quad x \in [0,\pi],
\end{eqnarray*}
where $c_1$, $c_2$, $c_3$, and $c_4$ are some coefficients. 
We can find $c_1 = 0$ and $c_4 = \frac{1}{4a^2}$ from the two boundary conditions (\ref{bc-Green}) at $x = 0^+$. The other two boundary conditions (\ref{bc-Green}) at $x = \pi$ gives the linear system for $c_2$ and $c_3$:
\begin{eqnarray*}
	\begin{bmatrix} \cosh(\pi a) \sin(\pi a) & \sinh(\pi a) \cos(\pi a) \\
		\sinh(\pi a) \cos(\pi a) & -\cosh(\pi a) \sin(\pi a)
	 \end{bmatrix}
 \begin{bmatrix} c_2 \\ c_3 \end{bmatrix} = 
 -c_4 \begin{bmatrix} \sinh(\pi a) \sin(\pi a) \\ \cosh(\pi a) \cos(\pi a) \end{bmatrix}.
\end{eqnarray*}
By Cramer's rule, we find the unique solution 
\begin{eqnarray*}
c_2 = -c_4 \frac{\sinh(2\pi a)}{\cosh(2\pi a) - \cos(2\pi a)}, \quad 
c_3 = c_4 \frac{\sin(2\pi a)}{\cosh(2\pi a) - \cos(2\pi a)},
\end{eqnarray*}
which results in the exact analytical expression
\begin{equation}
\label{G4exact-der}
G_{\mathbb{T}}'(x) = \frac{1}{4a^2} \frac{\sinh(ax) \sin a(2 \pi - x) - \sin(ax) 
	\sinh a(2\pi-x)}{\cosh(2\pi a) - \cos(2\pi a)}, \quad x \in [0,\pi].
\end{equation}
Integrating (\ref{G4exact-der}) in $x$ yields the exact analytical expression for $G_{\mathbb{T}}$:
\begin{equation}
\label{G4exact}
G_{\mathbb{T}}(x) = \frac{1}{8 a^3} 
\frac{g(x)}{\cosh(2\pi a) - \cos(2\pi a)}, \quad x \in [0,\pi],
\end{equation}
where 
\begin{eqnarray*}
g(x) &:=& \sinh(ax) \cos a(2 \pi - x) + \cosh(ax) \sin a(2\pi - x) \\
&& + \sin(ax) 
	\cosh a(2\pi-x) + \cos(ax) \sinh a(2\pi - x)
\end{eqnarray*}
and the constant of integration is set to zero due to the differential 
equation (\ref{alpha4-ode}).

We verify the validity of the exact solution (\ref{G4exact}) 
by comparing $G_{\mathbb{T}}(0)$ and $G_{\mathbb{T}}(\pi)$ with the Fourier series representation (\ref{green-fourier}) for $\alpha = 4$:
\begin{equation}
\label{G-4}
G_{\mathbb{T}}(0) = \frac{1}{2\pi} \sum_{n\in \mathbb{Z}}\frac{1}{4 a^4 + n^2}= \frac{1}{8a^3} \frac{\sinh(2\pi a) + \sin(2\pi a)}{\cosh(2\pi a) - \cos(2\pi a)}
\end{equation}
and
\begin{eqnarray}
G_{\mathbb{T}}(\pi) = \frac{1}{2\pi} \sum_{n\in \mathbb{Z}}\frac{(-1)^n}{4 a^4 + n^2}= \frac{1}{4 a^3} \frac{\sinh(\pi a) \cos(\pi a) + \sin(\pi a) \cosh(\pi a) }{\cosh(2\pi a) - \cos(2\pi a)}.
\label{G4value}
\end{eqnarray}
Indeed, the exact expressions coincide with those found from the numerical series (1) and (2) in \cite[Section 5.1.27]{Prudnikov}. 

It follows from (\ref{G4value}) that $G_{\mathbb{T}}(\pi)$ vanishes for $c = 4 a^4 > 0$ if and only if $a > 0$ is a solution of the transcendental equation 
\begin{equation}
\label{transcendental-eq}
\tanh(\pi a) + \tan(\pi a) = 0.
\end{equation}
Elementary graphical analysis on Figure \ref{transcend-roots} 
shows that there exist a countable sequence of zeros $\{ a_n \}_{n \in \mathbb{N}}$ such that 
$a_n \in \left(n - \frac{1}{4},n \right)$, $n \in \mathbb{N}$. Hence, $G_{\mathbb{T}}$ is not positive for $a \in (a_1,\infty)$.

\begin{figure}[h]
	\centering
	\caption{Countable sequence of zeros $\{ a_n \}_{n \in \mathbb{N}}$ of \ref{transcendental-eq}}
	\label{transcend-roots}
\end{figure}

Let us now show that the profile of $G_{\mathbb{T}}$ is strictly, monotonically decreasing on $(0,\pi)$ for small $a$. It follows from (\ref{G4exact-der}) that $G_{\mathbb{T}}'(x) < 0$ for $x \in (0,\pi)$ if and only if 
\begin{equation}
\label{graph-eq}
\frac{\sin(ax)}{\sinh(ax)} > \frac{\sin a(2\pi - x)}{\sinh a(2\pi -x)}, \qquad x \in (0,\pi).
\end{equation}
The function 
$$
x \mapsto \frac{\sin(ax)}{\sinh(ax)}
$$ 
is monotonically decreasing on $[0,2\pi]$ as long as 
\begin{equation}
\label{graph-eq-der}
\cos(ax) \sinh(ax) - \sin(ax) \cosh(ax) \leq 0, \qquad x \in [0,2\pi],
\end{equation}
which is true at least for $a \in (0,\frac{1}{2})$. 
Hence, $G_{\mathbb{T}}$ is strictly motonically decreasing on $(0,\pi)$ with $G_{\mathbb{T}}(\pi) > 0$ for $a \in (0,a_0)$, where $a_0 \in (\frac{1}{2},1)$.
On the other hand, it is obvious that there exists $a_* \in (1,\frac{3}{2})$ such that 
the inequality (\ref{graph-eq-der}) [and hence the inequality (\ref{graph-eq})] is violated at $x = \pi$ for $a \in (a_*,2)$, for which $G_{\mathbb{T}}'(x) > 0$ at least near $x = \pi$.  

The first part of the proposition is proven due to the relation $c = 4 a^4$. It remains to prove that $G_{\mathbb{T}}$ has a finite number of zeros on $\mathbb{T}$ for fixed $a \in [a_0,\infty)$ which becomes unbounded as $a \to \infty$. 
To do so, we simplify the expression (\ref{G4exact}) for $G_{\mathbb{T}}$ in the asymptotic limit of large $a$ for every fixed $x \in (0,\pi)$:
\begin{equation}
\label{G-simplified}
G_{\mathbb{T}}(x) = \frac{1}{8 a^3} \left[ e^{-ax} \cos(ax) + e^{-ax} \sin(ax) 
+ \mathcal{O}(e^{-a(2 \pi - x)})\right] \quad \mbox{\rm as} \quad a \to \infty.
\end{equation}
Thus, as $a$ gets large, there are finitely many zeros of $G_{\mathbb{T}}$ on $(0,\pi)$ but the number of zeros of $G_{\mathbb{T}}$ grows unbounded as $a \to \infty$.
\end{proof}

\begin{remark}
	The leading-order term in the asymptotic expansion (\ref{G-simplified}) represents Green's function $G_{\mathbb{R}}$. The proof of Conjecture \ref{conj-line} for $\alpha = 4$ follows from this explicit expression.
\end{remark}

\begin{remark}
Figure \ref{ineq-b8} shows boundaries on the $(a,x)$ plane between 
positive (yellow) and negative (blue) values of $G_{\mathbb{T}}$ (left) and $G_{\mathbb{T}}'$ (right). It follows from the figure that the zeros of $G_{\mathbb{T}}$ and $G_{\mathbb{T}}'$ are monotonically decreasing with respect to parameter $a$ and the number of zeros only grows as $a$ increases. In other words, zeros of $G_{\mathbb{T}}$ cannot coalesce and disappear. We were not able to prove these properties for every $a > 0$ inside $(0,\pi)$ (however, the proof can be given at $x = \pi$).
\end{remark}

\begin{figure}[h]
	\centering
	\includegraphics[width=0.49\linewidth]{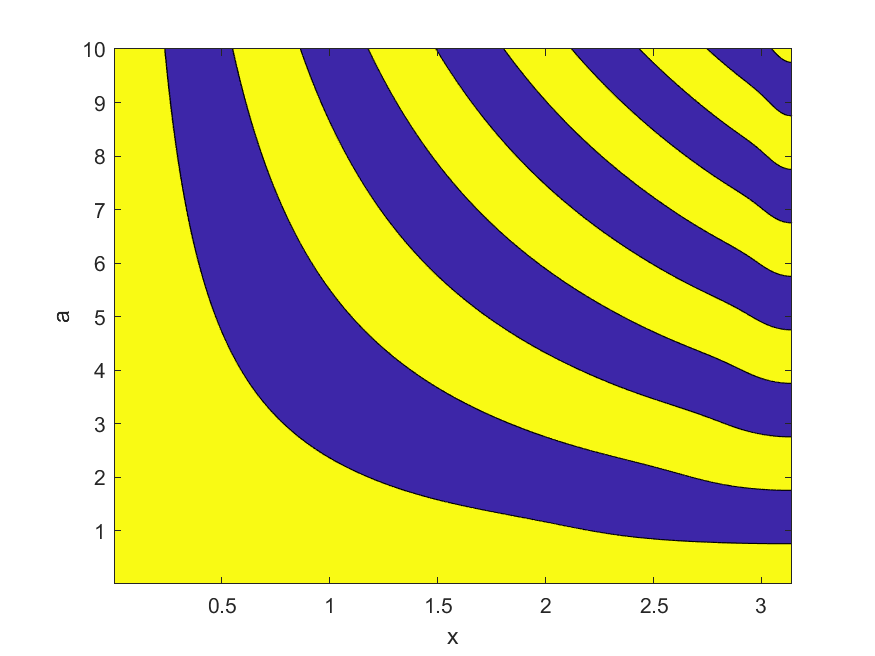}
	\includegraphics[width=0.49\linewidth]{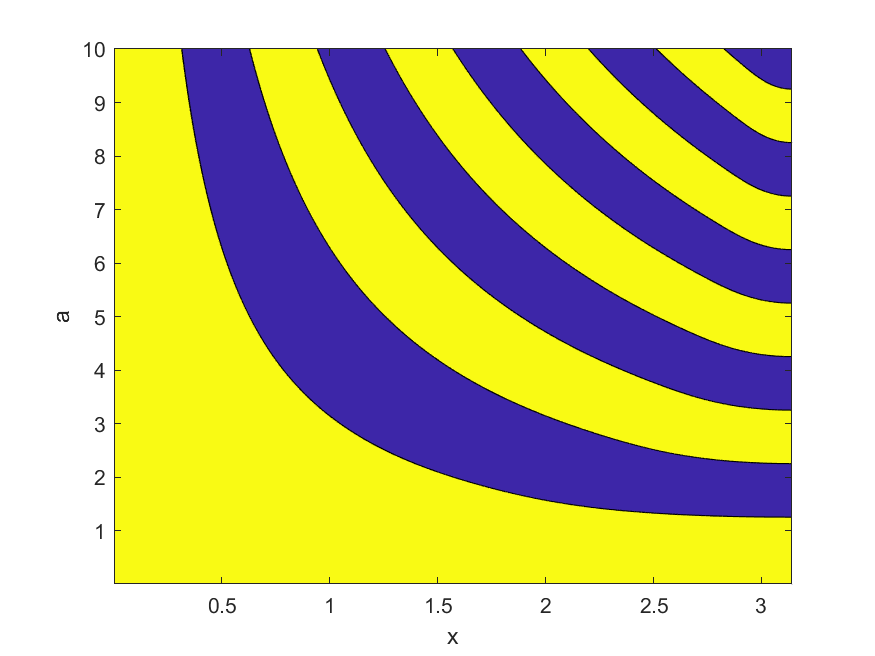}
	\caption{Left: areas on $(a,x)$ plane where $G_{\mathbb{T}}$ is positive (yellow) 
		and negative (blue). Right: the same but for $G_{\mathbb{T}}'$.}
	\label{ineq-b8}
\end{figure}

\section{Asymptotic approximation of the first zero of $G_{\mathbb{T}}(\pi)$}
\label{appendix-c}

Here, we obtain the formal asymptotic dependence of the first zero of $G_{\mathbb{T}}(\pi)$ as $(c,\alpha) \to (\infty,2)$. We use the integral 
representation (\ref{G-2}). Replacing the Mittag--Leffler function $E_{\alpha}$ by its leading-order asymptotic expression (\ref{mfl-asymp-2}) for $2 \leq \alpha < 6$, we obtain 
formally 
\begin{equation}
\label{G-3}
G_{\mathbb{T}}(\pi) = \frac{1}{2 \pi \alpha c } \left[  I(a,b) + \mbox{\rm error terms} \right],
\end{equation}
where 
\begin{equation}
\label{int-a-b}
I(a,b) := \int_0^{\infty} {\rm sech}^2\left(\frac{t}{2}\right) 
e^{a t} \cos(bt) dt
\end{equation}
with 
\begin{equation}
\label{ab-c}
a := c^{\frac{1}{\alpha}} \cos\left(\frac{\pi}{\alpha}\right), \quad 
b := c^{\frac{1}{\alpha}} \sin\left(\frac{\pi}{\alpha}\right).
\end{equation}
The limit $(c,\alpha) \to (\infty,2)$ such that $c < c_{\alpha}$ 
corresponds to the limit $b \to \infty$ with $a < 1$.

The integral $I(a,b)$ is the rapidly oscillating integral in the limit $b \to \infty$. We split it into two parts:
\begin{eqnarray}
\nonumber
I(a,b) &=& I_1(a,b) + I_2(a,b) \qquad \qquad\\
&=& \int_0^{\infty} {\rm sech}^2\left(\frac{t}{2}\right) \cosh(at) \cos(bt) dt + 
\int_0^{\infty} {\rm sech}^2\left(\frac{t}{2}\right) \sinh(at) \cos(bt) dt,
\label{int-split}
\end{eqnarray}
where $I_1(a,b)$ is exponentially small in $b$ and $I_2(a,b)$ is algebraically small in $b$. Indeed, by Darboux principle \cite{Boyd}, we evaluate 
the first integral for $a < 1$ with the residue theorem:
\begin{eqnarray*}
	I_1(a,b) && = \frac{1}{2} {\rm Re} \int_{-\infty}^{\infty} {\rm sech^2}\left(\frac{t}{2}\right) \cosh(at) e^{ibt} dt \qquad \qquad  \qquad \qquad   \qquad \qquad \\
	&& = {\rm Re} 4 \pi i {\rm Res}_{z = \pi i} \left[ \frac{e^z \cosh(az) e^{ibz}}{(1+e^z)^2} \right] + \mathcal{O}(e^{-3 \pi b}) \\
	\qquad \qquad 	&& = 4 \pi \left[ b \cos(\pi a) + a \sin(\pi a) \right] e^{-\pi b} + \mathcal{O}(e^{-3 \pi b}).
\end{eqnarray*}
Integrating the second integral by parts several times for $a < 1$ (see Section 5.2 in \cite{Olver}), we obtain 
\begin{eqnarray*}
	I_2(a,b) && = \frac{1}{b} \sin(bt) \sinh(at) {\rm sech}^2\left(\frac{t}{2}\right) \biggr|_{t=0}^{t \to \infty} - \frac{1}{b} \int_{0}^{\infty} \sin(bt) \frac{d}{dt} \left[ {\rm sech}^2\left(\frac{t}{2}\right) \sinh(at)\right] dt \qquad \qquad  \qquad \qquad   \qquad \qquad \\
	&& = \frac{\cos(bt)}{b^2} \frac{d}{dt} \left[ {\rm sech}^2\left(\frac{t}{2}\right) \sinh(at)\right] \biggr|_{t=0}^{t \to \infty} - \frac{1}{b^2} \int_{0}^{\infty} \cos(bt) \frac{d^2}{dt^2} \left[ {\rm sech}^2\left(\frac{t}{2}\right) \sinh(at)\right] dt \\
	\qquad \qquad 	&& = -\frac{a}{b^2} + \mathcal{O}\left(\frac{a}{b^4}\right).
\end{eqnarray*}
Finding zero of $I(a,b)$ in $a$ as $b \to \infty$ yields the approximation 
\begin{equation}
\label{ab-asympt}
a = 4\pi b^3 e^{-\pi b} \left[ 1 + \mathcal{O}\left(\frac{1}{b^2}\right) \right].
\end{equation}
Substituting (\ref{ab-c}) into (\ref{ab-asympt}) then taking the logarithm of both sides yields the following transcendental equation for $c^{\frac{1}{\alpha}}$
\begin{equation}
2\ln\left(c^{\frac{1}{\alpha}}\right) - \pi\sin\left(\frac{\pi}{\alpha}\right)c^{\frac{1}{\alpha}} + \mathcal{O}(c^{-\frac{2}{\alpha}}) = \ln\left(\frac{\cos(\frac{\pi}{\alpha})}{4 \pi \sin^{3}\left(\frac{\pi}{\alpha}\right)}\right).
\end{equation}
In order to determine the dependence of $c$ in terms of $\alpha$ we first factor $c^{\frac{1}{\alpha}}$ on the left hand side, then expand around $\alpha =2 $ to obtain
\begin{equation*}
- \pi c^{\frac{1}{\alpha}}\left[ 1 + \mathcal{O}\left((\alpha-2)^2\right) +\mathcal{O}\left( \frac{\ln\left(c^{\frac{1}{\alpha}} \right)}{c^{\frac{1}{\alpha}}}\right) \right]= \ln\left[\frac{\alpha-2}{16} + \mathcal{O}\left((\alpha-2)^2\right)\right].
\end{equation*}
Since $c^{\frac{1}{\alpha}}$ is of order $\mathcal{O}(\ln(\alpha -2))$, the above equation becomes
\begin{equation*}
-\pi c^{\frac{1}{\alpha}} = \ln\left(\frac{\alpha-2}{16}\right) \left[ 1 + 
\mathcal{O}\left(\frac{\ln|\ln(\alpha - 2)|}{|\ln(\alpha - 2)|} \right) \right],
\end{equation*}
which, as $\alpha \to 2$, implies
\begin{equation}\label{c-vs-alpha}
c = \frac{1}{\pi^2} \ln^{2}\left(\frac{\alpha-2}{16}\right) \left[1+\mathcal{O}\left(\frac{\ln|\ln(\alpha-2)|}{|\ln(\alpha-2)|}\right) \right].
\end{equation}
The asymptotic approximation (\ref{c-vs-alpha}) suggests that the first zero of $G_{\mathbb{T}}(\pi)$  
at $c = c_0(\alpha)$ satisfies $c_0(\alpha) \to \infty$ as $\alpha \to 2$.
However, we note that the asymptotic approximation of the root of $I(a,b)$ given by 
(\ref{c-vs-alpha}) is derived without analysis of the error terms 
in (\ref{G-3}).

\end{document}